\documentclass{amsart}

\usepackage{graphicx}
\usepackage[utf8]{inputenc}
\usepackage[english]{babel}
\usepackage{csquotes}
\usepackage{amssymb}
\usepackage{tikz-cd}

\usepackage{amsmath,amsfonts,pictex,graphicx,fullpage,etex,tikz,hyperref}
\usepackage{enumerate}
\usepackage{enumitem}

\newtheorem{theorem}{Theorem}[section]
\newtheorem{lemma}[theorem]{Lemma}
\newtheorem{proposition}[theorem]{Proposition}
\newtheorem{corollary}[theorem]{Corollary}

\theoremstyle{definition}
\newtheorem{definition}[theorem]{Definition}

\theoremstyle{remark}
\newtheorem{remark}[theorem]{Remark}
\newtheorem{notation}[theorem]{Notations}
\newtheorem{claim}[theorem]{Claim}

\numberwithin{equation}{section}

\makeatletter

\newcommand{\T}{\mathrm{T}}

\newcommand{\R}{\mathbb{R}}

\newcommand{\x}{\mathbf{x}}

\newcommand{\Iso}{\mathrm{Iso}}
\newcommand{\SO}{\mathrm{SO}}
\newcommand{\Hyp}{\mathbb{H}}
\newcommand{\SL}{\mathrm{SL}}
\newcommand{\dd}{\mathrm{d}}
\makeatother

\begin{document}

\title{Expanding curves in $\T^1(\Hyp^n)$ under geodesic flow and equidistribution in homogeneous spaces}

\author{Lei Yang $^{\ast}$}
\address{Mathematical Sciences Research Institute, Berkeley, CA, 94720, U.S.A.}

\curraddr{Einstein Institute of Mathematics, Hebrew University of Jerusalem, Jerusalem, 9190401, Israel}
\email{yang.lei@mail.huji.ac.il}
\thanks{$^{\ast }$ Supported in part by a Postdoctoral Fellowship at MSRI}

\subjclass[2010]{Primary 37A17; Secondary 22E40, 37D40}

\date{}


\date{}



\begin{abstract}
Let $H = \SO(n,1)$ and $A =\{a(t) : t \in \R\}$ be a maximal $\R$-split Cartan subgroup of $H$. Let $G$ be a Lie group containing 
$H$ and $\Gamma$ be a lattice of $G$. Let $x = g\Gamma \in G/\Gamma$ be a point of $G/\Gamma$ such that its $H$-orbit 
$Hx$ is dense in $G/\Gamma$. 
 Let $\phi: I= [a,b] \rightarrow H$ be an analytic curve, then $\phi(I)x$ gives an 
 analytic curve in $G/\Gamma$. In this article, we will prove the following result:
 if $\phi(I)$ satisfies some explicit geometric condition, then $a(t)\phi(I)x$ tends to be equidistributed in 
 $G/\Gamma$ as $t \rightarrow \infty$. It answers the first question asked by Shah in ~\cite{Shah_1} and 
 generalizes the main result of that paper.
\end{abstract}

\maketitle

\section{Introduction}
\subsection{Motivation}
\par Let $V$ be a $m$-dimensional hyperbolic space with finite volume, then $V$ can be written as $\mathbb{H}^m/\Gamma$, 
where $\Gamma < \Iso(\mathbb{H}^m) \cong \mathrm{SO}(m,1)$ is a lattice of $\SO(m,1)$. For $2 \leq n\leq m$, let 
$$\iota: \mathbb{H}^n \hookrightarrow \Hyp^m \rightarrow \Hyp^m/\Gamma $$
be a totally geodesic immersion from $\Hyp^n$ to $V$. Suppose the image $\iota(\Hyp^n)$ is dense in $V$. Now consider an analytic curve 
$$\phi: I=[a,b] \rightarrow \mathrm{T}^1(\mathbb{H}^n),$$
and its translates under geodesic flow $\{g_t: t>0\}$. 
Unless $\phi(I)$ is totally inside the contracting foliation of $\{g_t: t>0\}$, the length of $g_t(\phi(I))$ will increase exponentially as $t$ increases. Via the immersion $\iota$, 
we can embed $g_t(\phi(I))$ into $\mathrm{T}^1(V)$ and thus regard $g_t(\phi(I))$ as a curve in $\T^1(V)$. Then it is natural to ask whether the expanding curves $g_t(\phi(I))$ tend to be 
equidistributed in $\mathrm{T}^1(V)$, as $t \rightarrow +\infty$. In 2009, Nimish Shah ~\cite{Shah_1} answered this question in the following sense:
\begin{theorem}[see ~\cite{Shah_1}]
 \label{Shah_theorem_1}
\par There is a naturally difined visual map 
\begin{equation}
\label{equ_visual_map}
\mathrm{Vis}: \mathrm{T}^{1}(\mathbb{H}^n) \rightarrow \partial \mathbb{H}^n \cong \mathbb{S}^{n-1},
\end{equation}
which sends every vector $v \in \mathrm{T}^1(\mathbb{H}^n)$ to the ideal boundary $\partial \mathbb{H}^n$ through the 
geodesic ray starting from $v$. If $\mathrm{Vis}(\phi(I))$ is not contained in a proper subsphere of $\partial \mathbb{H}^n$, 
then the expanding curves, $\{g_t(\phi(I)): t >0\}$, will tend to be equidistributed in $\mathrm{T}^1(V)$ as 
$t \rightarrow +\infty$. In other words, for any continuous, compactly supported function $f \in C_c(\mathrm{T}^1(V))$, 
$$\lim_{t \rightarrow +\infty} \frac{1}{b-a} \int_{a}^b f(g_t(\phi(s))) \dd s = \int_{\mathrm{T}^1(V)} f(v)\dd \mathrm{Vol}(v).$$
\end{theorem}
In ~\cite{Shah_1}, the theorem is proved using ergodic theory of homogeneous dynamics. It turns out that the above theorem is equivalent 
to the following theorem in homogeneous dynamics:
\begin{theorem}
\label{Shah_theorem_2}
Let $H = \mathrm{SO}(n,1)$, and $G= \mathrm{SO}(m,1)$, where $m\geq n \geq 2$. Fix an 
embedding $H \hookrightarrow G$. Let $A=\{a(t): t \in \R\}$ be a maximal $\R$-split Cartan subgroup of $H$ (it is one-parameter subgroup 
because $\R$-rank of $H$ is $1$). Let $\Gamma < G$ be a lattice of $G$, then $G/\Gamma$ admits a $G$-invariant probability measure, denoted by 
$\mu_G$. Fix a point $x = g\Gamma \in G/\Gamma$ such that its $H$-orbit $Hx$ is dense in $G/\Gamma$. Let $\phi: I=[a,b] \rightarrow H$ be an 
analytic curve such that $\phi(I)$ is expanded by $a(t)$ for $t >0$. If the image of $\phi(I)$ under the visual map 
$$\mathrm{Vis}': H \rightarrow \partial \mathbb{H}^n$$ 
is not contained in a proper subsphere of $\partial \mathbb{H}^n \cong \mathbb{S}^{n-1}$, then as 
$t\rightarrow +\infty $, $a(t)\phi(I)x$ tends to be equidistributed in $G/\Gamma$, i.e., for 
any $f \in C_c(G/\Gamma)$, 
$$\lim_{t\rightarrow +\infty} \frac{1}{b-a}\int_a^b f(a(t)\phi(s)x)\dd s = \int_{G/\Gamma} f \dd\mu_G .$$
\end{theorem}
\begin{remark}
\label{remark_1}
 \par For $H=\mathrm{SO}(n,1)$, the visual map $\mathrm{Vis}': H \rightarrow \partial \mathbb{H}^n$ is defined as follows. Note that there is a natural action of $H$ on the unit tangent bundle of
$n$-dimensional hyperbolic space $\T^1(\mathbb{H}^n)$. For a fixed vector 
$v_0 \in \T^1(\mathbb{H}^n)$, we denote by $M$ the stabilizer of $v_0$, it is well 
known that $M \cong \mathrm{SO}(n-1)$. Therefore $\T^1(\mathbb{H}^n) \cong M\setminus H$. 
We can choose $v_0$ such that $M \subset Z_H(A)$, where $Z_H(A)$ denotes the centralizer of $A$ in $H$. 
Under this identity,
the action of geodesic flow $\{g_t: t \in \R\}$ on $\T^1(\mathbb{H}^n)$ is the same as the multiplication action of 
$A = \{a(t): t \in \R\}$ on $M\setminus H$, i.e., 
for $v = Mh \in \T^1(\mathbb{H}^n)$, $g_t(v) = a(t)M h = M a(t) h$. Now let $\tau:
H \rightarrow \T^1(\mathbb{H}^n) \cong M\setminus H$ denote the canonical
projection, then we could define the visual map 
$$\mathrm{Vis}'=\mathrm{Vis}\circ \tau: H \rightarrow \partial \mathbb{H}^n.$$ 
\end{remark}

\subsection{Main result}

\par The proof of Theorem \ref{Shah_theorem_2} makes use of Ratner's theorem concerning the classification of probability measures invariant 
under unipotent flows, and the linearization technique. In ~\cite{Shah_1}, most of the argument works if one replaces $\mathrm{SO}(m,1)$ by arbitrary 
Lie group $G$ containing $H= \mathrm{SO}(n,1)$. But to prove that if $a(t)\phi(I)x$ does not 
tend to be equidistributed, then $\mathrm{Vis}'(\phi(I))$ must be contained in a proper
subsphere of $\partial \mathbb{H}^n$, the argument heavily depends on the group structure of 
$G= \mathrm{SO}(m,1)$. In ~\cite{Shah_1}, it was conjectured that the same result will still hold if one replaces 
$\mathrm{SO}(m,1)$ by arbitrary Lie group $G$ containing $H$. It was also selected as an unsolved conjecture in 
Gorodnik's survey ~\cite{Gorodnik} (see ~\cite[Conjecture 19]{Gorodnik}). The main purpose of 
this article is to prove this conjecture. Our main result is the following:
\begin{theorem}
\label{goal_thm} 
Let $H = \SO(n,1)$ where $n \geq 2$, and let $A = \{a(t): t \in \R\}$ be a maximal $\R$-split Cartan subgroup of $H$. Let $G$ be a Lie group containing $H$ and let $\Gamma < G$ be a lattice of $G$. Then $G/\Gamma$ admits a $G$-invariant probability measure, denoted by $\mu_G$. Let $x = g \Gamma$ be a point such that $Hx$ is dense in $G/\Gamma$. Let $\phi: I =[a,b] \rightarrow H$ be an analytic curve. If $\mathrm{Vis}'(\phi(I))$ is not contained in a proper subsphere of $\partial \mathbb{H}^n \cong \mathbb{S}^{n-1}$, then $a(t)\phi(I)$ tends to be equidistributed 
as $t \rightarrow \infty$. In other words, for any $f\in C_c(G/\Gamma)$,
$$
\lim_{t\rightarrow \infty} \frac{1}{|b-a|}\int_{a}^b f(a(t) \phi(s)x)\dd s
=\int_{G/\Gamma} f \dd\mu_G.
$$
\end{theorem} 

\begin{remark}
 The assumption that $Hx$ is dense in $G/\Gamma$ does not harm the generality of the theorem. In fact, since $\mathrm{SO}(n,1)$ is generated by its one parameter unipotent 
 subgroups, by Ratner's theorem (see ~\cite{Ratner_2}), 
 the closure of $Hx$ in $G/\Gamma$ is homogeneous, i.e., there exists some closed subgroup $F$ 
 of $G$ containing $H$ such that $F\cap g\Gamma g^{-1}$ is a lattice of $F$ and moreover 
 the closure $\overline{Hg\Gamma} = Fg\Gamma$. It implies that $H g\Gamma g^{-1}$ is dense in 
 $F g\Gamma g^{-1}$. Since $F g\Gamma g^{-1} \cong F/F\cap g\Gamma g^{-1}$, we will 
 have $H (F\cap g\Gamma g^{-1})$ is dense in $F/F\cap g\Gamma g^{-1}$. 
 Then we can replace $G$ by $F$, $\Gamma$ by $F\cap g\Gamma g^{-1}$ 
 and $g$ by $e$ to make the $H$-orbit dense.
\end{remark}
\subsection{Related results}
\par The study of limit distribution of evolution of curves on 
homogeneous spaces under some diagonal flow was initiated by Nimish Shah 
in several papers: ~\cite{Shah_1}, ~\cite{Shah_2}, ~\cite{Shah_3} and ~\cite{Shah_4}.  
\par ~\cite{Shah_2} proves that the equidistribution result in ~\cite{Shah_1} also holds for $C^n$-smooth curves.

\par ~\cite{Shah_3} studies the action of $\mathrm{SL}(n+1,\R)$ on a homogeneous space $G/\Gamma$. Suppose $x =g\Gamma \in G/\Gamma$ has 
dense $\mathrm{SL}(n+1,\R)$-orbit in $G/\Gamma$. Given the diagonal flow
$$A:= \left\{a(t):=\begin{bmatrix}e^{nt} & \\ & e^{-t} \mathrm{I}_n\end{bmatrix}: t\in \R \right\},$$
and an analytic curve 
$$\phi: I=[a,b]\rightarrow \mathrm{SL}(n+1,\R),$$
the equidistribution result for $a(t)\phi(I)x$ (as $t\rightarrow +\infty$) is established assuming 
some geometric condition on $\phi(I)$.
\par ~\cite{Shah_4} considers some more general diagonal flow in $\mathrm{SL}(n+1,\R)$ and establishes a similar equidistribution result. 
\par Both the results of ~\cite{Shah_3} and ~\cite{Shah_4} have interesting applications to Diophantine approximation, see ~\cite{Dani_2}, ~\cite{Kleinbock_Weiss},
~\cite{Shah_3} and ~\cite{Shah_4} for details.

\par To generalize the result of ~\cite{Shah_1} to general Lie group 
$G$, one has to avoid the original group theoretic argument, and get the geometric condition 
($\mathrm{Vis}'(\phi(I))$ is contained in a proper subsphere) only from the representation of $H$ on a finite dimensional 
vector space $V$. This is accomplished in this article based on a new observation on the representation of $\mathrm{SL}(2,\R)$. It will be 
proved in Section \ref{basic lemma proof}.
\par This article is organized as follows: 
\begin{itemize}
 \item In Section \ref{preliminary}, we recall some basic facts concerning the structure of 
$H=\mathrm{SO}(n,1)$ and geodesic flow on $\T^1(\mathbb{H}^n) \cong M\setminus H$, 
and give a basic reduction of the original problem.
\item In Section \ref{limit measure invariant unipotent}, 
we follow the argument 
in ~\cite{Shah_1} to show that the limit measure of the evolutions of the normalized measure on our original 
curve under the action of geodesic flow is also a probability measure on $G/\Gamma$ and is invariant under some 
unipotent subgroup.
\item In Section \ref{linearization}, we use the linearization technique developed in ~\cite{Shah_1} to show that if the equidistribution fails, then the curve satisfies a linear algebraic condition concerning a 
particular representation of $\mathrm{SO}(n,1)$ on a finite dimensional vector space $V$. 
\item In Section \ref{basic lemma proof}, we will prove a technical 
result concerning representations of $\mathrm{SL}(2,\R)$. It is crucial in the proof of Theorem \ref{goal_thm}.
\item In Section \ref{conclusion}, we combine the results we proved in Section 
\ref{linearization} and Section \ref{basic lemma proof} to complete the proof of Theorem \ref{goal_thm}.
\end{itemize}
\begin{notation}
 In this article, we will use the following notations: for $\epsilon> 0$ small, 
 and two quantities $A$ and $B$, $A \overset{\epsilon}{\approx} B$ means that 
 $|A-B| \leq \epsilon$. Fix a right $G$-invariant metric $d(\cdot, \cdot)$ on $G$. For $x_1, x_2 \in G/\Gamma$,
 and $\epsilon >0$, $x_1 \overset{\epsilon}{\approx} x_2$ means $x_2 =g x_1$ such that $d(g,e)< \epsilon$.
\end{notation}

\noindent {\em Acknowledgements.} I would like to express my deep gratitude to my advisor, Professor Nimish Shah, for suggesting this problem to me, and 
his continuous advise and support during the process of the work. I also would like to thank Professor Kleinbock for reading
my manuscript and giving me many valuable comments and suggestions to improve the exposition of the article.  Thanks are also owed the referee 
for many useful suggestions.


\section{Preliminaries and basic reduction}
\label{preliminary}

\subsection{Preliminaries concerning the structure of $\SO(n,1)$}
We realize $H=\mathrm{SO}(n,1)$ as the group of $n+1$ by $n+1$ matrices with determinant one and preserving the
quadratic form $Q$ in $(n+1)$ real variables defined as follows:
$$
Q(x_0,x_1,\dots , x_n)=2x_0 x_n -(x_1^2 + \dots + x_{n-1}^2).
$$
It is easy to check that $Q$ is of signature $(n,1)$. 
Let 
$$A:=\left\{ a(t) := \begin{bmatrix} e^t & & & & \\ & 1 & & & \\ & & \ddots & & \\
& & & 1 & \\ & & & & e^{-t} \end{bmatrix} : t \in \R \right\},$$ 
then $A$ is a maximal $\R$-split one parameter Cartan subgroup of $H$. Let $A^+ :=\{a(t): t > 0\}$. Let $M \cong \SO(n-1)$ be defined as in Remark \ref{remark_1}, then by our discussion there, for $v = Mh \in M\setminus H \cong \T^1(\mathbb{H}^n)$, the
orbit of $v$ under geodesic flow $\{g_t : t\in \R\}$ is the orbit
under the action of $A$, say $\{M a(t) h: t \in \R \}$. Define the character 
$$\alpha: A \rightarrow \mathbb{R}^+$$ 
by $\alpha(a(t))=e^{t/2}$, then $A^+=\{a\in A: \alpha(a)>1\}$. Let
$K \cong \mathrm{SO}(n)$ be a maximal compact subgroup of $H$ such that $M=Z_H(A)\cap K$,
where $Z_H(A)=MA$ denotes the centralizer of $A$ in $H$. Then in the above realization,
$$
M = \left\{ m=\begin{bmatrix} 1 &  &  \\  &
k(m) &  \\  &  & 1\end{bmatrix} : k(m) \in \mathrm{SO}(n-1) \right\}.
$$
Here $k:M\rightarrow \mathrm{SO}(n-1)$ gives an isomorphism between $M$ and $\SO(n-1)$. 
\par Define:
$$
N^-:=\{h\in H: a^k h a^{-k}\rightarrow e \text{ as } k\rightarrow
+\infty \text{ for any } a\in A^+\},$$
and
$$
N:=\{h\in H: a^{-k}h a^k \rightarrow e \text{ as } k\rightarrow
+\infty \text{ for any } a \in A^+\}.
$$
Then $P^-:=MAN^-$ is a minimal parabolic subgroup of $H$. It is well known that $N$ is 
an abelian Lie subgroup of $H$. Let
$\mathfrak{n}$ denote the Lie algebra of $N$, then
$\mathfrak{n}\cong \mathbb{R}^{n-1}$.
Let $u: \mathfrak{n}\cong \mathbb{R}^{n-1}\rightarrow
N$ denote the exponential map from $\mathfrak{n}$ to $N$, then for $\mathbf{x}=(x_1, x_2,\dots, x_{n-1}) \in \R^{n-1} \cong \mathfrak{n}$,
$$u(\mathbf{x})=\begin{bmatrix} 1 & x_1 & \dots & x_{n-1} & \|\mathbf{x}\|^2/2
\\  &  1 & & & x_1 \\ & & \ddots &   & \vdots \\ & & & 1 & x_{n-1}
\\ & & & & 1\end{bmatrix} .$$
Similarly, for $\mathfrak{n}^- \cong \R^{n-1}$, the Lie algebra of $N^-$, one could
define 
$$u^-:  \mathfrak{n}^- \cong \mathbb{R}^{n-1} \rightarrow N^-.$$ Moreover, for
$\mathbf{x} \in \R^{n-1} \setminus \mathbf{0}$, there exists a subgroup of $H$, isomorphic to $\SL(2,\R)$, that contains the one-parameter unipotent subgroup
$\{u(t\mathbf{x}): t\in \mathbb{R}\}$ and the diagonal
subgroup $A$. Let us denote this subgroup by $\SL(2,\mathbf{x})$. In $\SL(2, \mathbf{x})$, $u(r\mathbf{x})$
corresponds to 
$$\begin{bmatrix} 1& r \\ 0 &
1\end{bmatrix},$$ 
$a(t)$ corresponds to 
$$\begin{bmatrix}
e^{t/2} & 0 \\ 0 & e^{-t/2}
\end{bmatrix},$$
and $u^{-}(r \x / \|\x\|^2 )$
corresponds to 
$$\begin{bmatrix}1 & 0 \\ r & 1 \end{bmatrix}.$$

\par Let $M$ act on $\mathbb{R}^{n-1}$
by the natural action of $\SO(n-1)$ on $\R^{n-1}$, then it is easily seen that 
for all $z \in M$ and $\mathbf{x} \in \R^{n-1}$, $u(z \mathbf{x})= z u(\mathbf{x}) z^{-1}$. 

\subsection{Basic reduction}
\par We first reduce Theorem \ref{goal_thm} to the following theorem:
\begin{theorem}
\label{main_theorem} 
Let $\varphi: I=[a,b] \rightarrow
\mathbb{R}^{n-1}$ be an analytic curve which is not contained in any
proper subsphere or proper affine subspace of $\R^{n-1}$. 
Let $x = g\Gamma \in G/\Gamma$ such that $Hx$ is dense in $G/\Gamma$. Then $a(t) u(\varphi(I))x$ tends to be equidistributed in $G/\Gamma$ as $t \rightarrow +\infty$, i.e., for any
$f\in C_c(G/\Gamma)$,
$$
\lim_{t\rightarrow \infty} \frac{1}{|I|}\int_{I} f(a(t)
u(\varphi(s))x) \dd s=\int_{G/\Gamma} f \dd \mu_G.
$$
\end{theorem}

\begin{proof}[Proof of Theorem \ref{goal_thm} assuming Theorem \ref{main_theorem}]
\par Let $\phi: I \rightarrow H := \mathrm{SO}(n,1)$ be
an analytic curve, then by the canonical decomposition $H = N^- Z_H(A) N$, there exists some $\varphi: I \rightarrow
\mathbb{R}^{n-1}$ such that $\phi(s)= N^{-}(s)K(s)u(\varphi(s))$
where $N^{-}(s) \in N^-$, and $K(s)\in Z_H(A) = AM$. One can easily check that $\mathrm{Vis}(\phi(I))$ is
contained in a proper subsphere if and only if $\varphi(I)$ is
contained in a proper subsphere or hyperplane. Then for any $f \in
C_c(G/\Gamma)$, the normalized line integral of $f$ on $\phi(I)$ can be written as:
$$\begin{array}{rl} & \frac{1}{|I|}\int_{I} f(g_t \phi(s)x)\dd s
\\ = & \frac{1}{|I|}\int_{I} f(a(t) N^{-}(s) K(s) u(\varphi(s))x)\dd s \\ 
= & \frac{1}{|I|}\int_{I} f(a(t) N^{-}(s)a(-t) K(s) a(t) u(\varphi(s))x)\dd s.
\end{array}$$
Since $N^{-}(I)\subset N^-$ is compact, $a(t)
N^-(I) a(-t) \rightarrow e$ as $t\rightarrow +\infty$. Then for
any given $\epsilon > 0$, there exists some $t_0$, such that for
$t> t_0$ large enough, $f(a(t) N^{-}(s)a(-t) K(s) a(t) u(\varphi(s))x)
\overset{\epsilon}{\approx} f(K(s)a_t u(\varphi(s))x)$ for all
$s\in I$. This means $\int_{I} f(g_t \phi(s)x)ds \overset{\epsilon
|I|}{\approx} \int_{I} f(K(s)a_t u(\varphi(s))x)ds$.
\par By our assumption, $\varphi(I)$ is not
contained in any proper subsphere or hyperplane, then the same holds for any subinterval $J \subset I$, since $\varphi$
is analytic. Then by Theorem \ref{main_theorem}, $\frac{1}{|J|} \int_{J} f(a(t)
u(\varphi(s))x)\dd s \rightarrow \int_{G/\Gamma} f(x) \dd\mu_G$ as
$t\rightarrow +\infty$ for any subinterval $J\subset I$. Since $f \in
C_c(G/\Gamma)$, we may divide $I$ into several subintervals
$J_1,J_2, \dots , J_k$ such that
$f(K(s_0)K^{-1}(s)x)\overset{\epsilon}{\approx} f(x)$ for any $x \in G/\Gamma$ and any
$s_0, s$ in the same subinterval. Since for each $J_i$, the
normalized line integral of $f$ on $J_i$ tends to $\int_{G/\Gamma} f(x) \dd\mu_G$, we
have that there exists some $t_1>0$ such that $\int_{J_i} f(a(t)
u(\varphi(s))x) \dd s \overset{\epsilon |J_i|}{\approx} |J_i| \int_{G/\Gamma}
f \dd\mu_G$ for all $J_i$ and $t> t_1$. For fixed $J_i$ and a fixed point
$s_0 \in J_i$, let us define 
$$f_0 (x):= f(K(s_0)x) \text{ for } x \in G/\Gamma.$$ Then for $t > t_1$,
$$
\begin{array}{rl}
\int_{J_i} f(K(s)a(t) u(\varphi(s))x) \dd s & \overset{\epsilon
|J_i|}{\approx} \int_{J_i} f(K(s_0) a(t) u(\varphi(s))x)\dd s \\
 & =\int_{J_i} f_0( a(t) u(\varphi(s))x)\dd s \\
 & \overset{\epsilon|J_i|}{\approx} |J_i|\int_{G/\Gamma} f_0
 \dd \mu_G \\ & = |J_i| \int_{G/\Gamma} f \dd\mu_G.
\end{array}
$$
The last equality holds since $\mu_G$ is $G$-invariant, and in particular, is invariant under the action of $K(s_0)$. By repeating the above argument for $i=1,2,\dots, k$ and 
summing all these approximations, we have that
$$
\int_{I} f(K(s)a(t) u(\varphi(s))x)\dd s \overset{2 \epsilon
|I|}{\approx} |I| \int_{G/\Gamma}f \dd\mu_G,
$$
i.e., 
$$\frac{1}{|I|}\int_{I} f(K(s)a(t)
u(\varphi(s))x)\dd s \overset{2 \epsilon }{\approx} \int_{G/\Gamma}f
\dd\mu_G.
$$
Combined with the fact that $f(a(t)\phi(s)x)\overset{\epsilon}{\approx}
f(K(s)a(t) u(\varphi(s))x)$ for $t> t_0$, this implies that for $t > \max\{t_0,
t_1\}$,
$$
\frac{1}{|I|}\int_{I} f(a(t) \phi(s)x)\dd s \overset{3 \epsilon
}{\approx} \int_{G/\Gamma}f \dd\mu_G.
$$
  This completes the proof of Theorem \ref{goal_thm} since the above holds for arbitrary $\epsilon >0$.
 
\end{proof}
The main part of this article is devoted to the proof of Theorem \ref{main_theorem}.


\section{Non-divergence and unipotent invariance of limit measures}
\label{limit measure invariant unipotent}

\subsection{Modification of the original curve}
\par 
Let $\varphi: I=[a,b] \rightarrow \R^{n-1}$ denote an analytic curve.
For $t > 0$, we define $\lambda_t$ to be the normalized Lebesgue measure on 
the curve $a(t)u(\varphi(s))x$, i.e., for any continuous compactly support function 
$f\in C_c(G/\Gamma)$, 
$$\lambda_t(f) := \frac{1}{|I|}\int_a^b f(a(t)u(\varphi(s))x)\dd s.$$
\par Our goal is to show that as $t \rightarrow +\infty$, $\lambda_t \rightarrow \mu_G$. Our basic tool is Ratner's theorem 
(see ~\cite{Ratner} and ~\cite{Ratner_2}). To apply Ratner's theorem, we need to find a unipotent subgroup which preserves limit points of $\{\lambda_t : t >0\}$. To do this, we modify the curve $u(\varphi(I))$ as follows. Fix a vector $\mathbf{e}_1 =(1,0,\dots, 0) \in \R^{n-1}$. Since $\varphi'(s) \neq \mathbf{0}$
for all $s \in I$, there exists an analytic curve $z: I \rightarrow M$ such that 
$$z(s)u(\varphi'(s))z^{-1}(s) = u(z(s)\varphi'(s)) = u(\mathbf{e}_1) \text{ for all } s \in I .$$
 Let $\mu_t$ denote the normalized Lebesgue measure on 
the modified curve $\{z(s)a(t)u(\varphi(s))x: s \in I\}$, i.e., for $f \in C_c(G/\Gamma)$,
$$\mu_t(f) := \frac{1}{b-a}\int_a^b f(z(s)a(t)u(\varphi(s))x)\dd s.$$

\begin{remark}
$\quad$
\par 1. The construction of the modified measure $\mu_t$ is due to Nimish Shah ~\cite[Section 3]{Shah_1}.
 \par 2. For any subinterval $J \subset I$, we define $\mu^{J}_t$ (and $\lambda^{J}_t$, respectively) 
 to be the normalized Lebesgue measure on the curve $\{z(s) a(t) u(\varphi(s)): s \in J\}$ (and on $a(t)u(\varphi(J))$, respectively), i.e., for $f\in C_c(G/\Gamma)$,
 $$\mu^J_t(f) := \frac{1}{|J|} \int_J f(z(s)a(t)u(\varphi(s))x)\dd s,$$
 and
 $$\lambda^J_t(f) := \frac{1}{|J|} \int_J f(a(t)u(\varphi(s))x)\dd s.$$
\end{remark}

\begin{proposition}
  Suppose for any subinterval 
 $J \subset I$, 
 $\mu^{J}_t \rightarrow \mu_G$ as $t \rightarrow +\infty$, then the same holds for $\lambda_t$, i.e., $\lambda_t \rightarrow \mu_G$ as 
 $t \rightarrow +\infty$.
\end{proposition}

\begin{proof}
The proof is similar to the proof of Theorem \ref{goal_thm} assuming Theorem \ref{main_theorem}.
\par For any $f \in C_c(G/\Gamma)$, and $\epsilon >0$, one can divide $I$ into small subintervals 
$J_1, J_2, \dots, J_k$, such that for each $J_i$ and $s_1, s_2 \in J_i$, $f(z^{-1}(s_1)z(s_2)x) \overset{\epsilon}{\approx} f(x)$ for any 
$x \in G/\Gamma$. 
\par For a fixed subinterval $J_i$, let us fix a point $s_0 \in J_i$, and define 
$f_0(x) = f(z^{-1}(s_0)x)$. Then for any $s \in J_i$, 
$$f(a(t)u(\varphi(s))x) \overset{\epsilon}{\approx} f_0(z(s)a(t)u(\varphi(s))x),$$
therefore 
$$\lambda^{J_i}_t(f) \overset{\epsilon}{\approx} \mu^{J_i}_t(f_0),$$
for any $t >0$. Since $\mu^{J_i}_t(f_0) \rightarrow \mu_G(f_0)$, there exists a constant $T_i >0$,
such that for $t > T_i$, 
$$\mu^{J_i}_t (f_0) \overset{\epsilon}{\approx} \mu_G(f_0).$$
On the other hand, since $\mu_G$ is $G$-invariant, and $f_0(x) = f(z^{-1}(s_0)x)$ is a left 
translation of $f(x)$, we have 
$$\mu_G(f_0) = \mu_G(f).$$
Therefore, for $t >T_i$,
$$\lambda^{J_i}_t(f) \overset{2\epsilon}{\approx} \mu_G(f).$$
By repeating the above argument for all $i=1,2,\dots,k$ and summing the above approximations for 
all $i=1,2,\dots , k$, we get that for $t >\max_{1\leq i \leq k} \{T_i\}$, 
$$\lambda_t(f) \overset{2\epsilon}{\approx} \mu_G(f).$$
The above holds for arbitrary $\epsilon >0$ and $f \in C_c(G/\Gamma)$, so this proves that
$\lambda_t \rightarrow \mu_G$ as $t \rightarrow +\infty$.
\end{proof}
\par This proposition allows us to study $\mu_t$ instead of $\lambda_t$.

\subsection{Non-divergence of limit measures}
\par We first show that every limit point $\mu_{\infty}$ of $\{\mu_t : t >0\}$ is still 
a probability measure.
\begin{proposition}
\label{prop_no_escape_mass}
 For any $\epsilon >0$, there exists a compact subset $\mathcal{K}_{\epsilon} \subset G/\Gamma$ such that 
 $\mu_t(\mathcal{K}_{\epsilon}) \geq 1-\epsilon$ for all $t >0$.
\end{proposition}
\par The proposition is due to N. Shah ~\cite{Shah_1}. We modify the original proof 
a little bit to fit our setup.
\begin{definition}
\label{def_representation_G}
 Let $\mathfrak{g}$ denote the Lie algebra of $G$, and let $d := \dim G$. Define
 $$V := \bigoplus_{i=1}^{d} \bigwedge^i \mathfrak{g}.$$
 Fix a norm $\|\cdot\|$ on $V$. Let $G$ act on $V$ via $\bigoplus_{i=1}^d \bigwedge^i \mathrm{Ad}(G)$. This defines a linear representation of $G$:
 $$G \rightarrow \mathrm{GL}(V).$$
\end{definition}
The following theorem due to Dani and Kleinbock-Margulis is the basic tool to prove Proposition \ref{prop_no_escape_mass}.
\begin{theorem}[see ~\cite{Dani} and ~\cite{Kleinbock_Margulis}]
 \label{non_divergence_theorem}
There exist finitely many vectors  $v_1, v_2, \dots , v_r \in V$ such that for each 
 $i=1,2,\dots, r$, the orbit $\Gamma v_i$ is discrete. Moreover, for any $\epsilon >0$ and $R > 0$, there exists a
compact set $K\subset G/\Gamma$ such that for any $t >0$ and any subinterval $J\subset I$, one of the
following holds:
\begin{enumerate}[label=\textbf{A.\arabic*}]
\item There exist $\gamma \in \Gamma$ and $j\in \{1,\dots , r\}$
such that $$\sup_{s\in J} \| a(t)u(\varphi(s)) g \gamma v_j \| < R,$$
\item $$|\{ s\in J:  a(t)u(\varphi(s))x \in K\}| \geq (1-\epsilon)|J|.$$
\end{enumerate}
\end{theorem}
\begin{remark}
 For the case when $\varphi(s)$ is a polynomial curve, the proof is due to Dani ~\cite{Dani}, and for the case of analytic curves, 
 the proof is due to Kleinbock and Margulis ~\cite{Kleinbock_Margulis}. The crucial part of the proof is to show that there exist constants $C>0$ and $\alpha>0$ such that for any $t >0$, all the coordinate functions of $a(t)u(\varphi(\cdot))$ are $(C, \alpha)$-good. Here a function $f: I \rightarrow \R$ is called
 $(C,\alpha)$-good if for any subinterval $J \subset I$ and any $\epsilon >0$, the following holds:
 $$|\{s\in J: |f(s)|<\epsilon\}| \leq C\left(\frac{\epsilon}{\sup_{s\in J}|f(s)|}\right)^{\alpha} |J|.$$
 See ~\cite{Kleinbock_Margulis} for details.
\end{remark}

\begin{definition}
 Let $F$ be a Lie group, and $V$ be a finite dimensional linear representation of $F$. Then for a one-parameter diagonalizable subgroup 
 $A=\{a(t): t \in \R\}$ of $F$, we could decompose $V$ as direct sum of eigenspaces of $A$, i.e.,
 $$V = \bigoplus_{\lambda \in \R} V^{\lambda}(A),$$
 where $V^{\lambda}(A) = \{v\in V: a(t)v = e^{\lambda t} v\}$.
 \par We define
 $$V^{+}(A) = \bigoplus_{\lambda >0} V^{\lambda}(A),$$
 $$V^{-}(A) = \bigoplus_{\lambda <0 } V^{\lambda}(A),$$
 and similarly,
 $$V^{+0}(A) = V^{+}(A) + V^0(A),$$
 $$V^{-0}(A) = V^{-}(A) + V^0(A) .$$
 For a vector $v \in V$, we denote by $v^{+}(A)$ ($v^{-}(A)$, $v^0(A)$, $v^{+0}(A)$ and $v^{-0}(A)$ respectively) the projection of $v$ onto $V^{+}(A)$ ($V^{-}(A)$, $V^0(A)$, $V^{+0}(A)$ and 
 $V^{-0}(A)$ respectively).
\end{definition}

The proof of Proposition \ref{prop_no_escape_mass} is based on the following basic lemma on representations of $\mathrm{SL}(2,\R)$, 
due to N. Shah:
\begin{lemma}[See Lemma 2.3 of ~\cite{Shah_1}]
\label{lemma_non_divergent}
\par Let $V$ be a representation of $\mathrm{SL}(2,\R)$, fix a norm $\|\cdot\|$ on $V$.
We define 
$$A := \left\{a(t):= \begin{bmatrix}e^{t} & \\ & e^{-t}\end{bmatrix}: t \in \R \right\},$$
and
$$U^{+}(A) := \left\{u(r):=\begin{bmatrix}1 & r \\ 0 & 1 \end{bmatrix}: r \in \R \right\}.$$
Then for any $r >0$, there exists a
constant $\kappa=\kappa(r)> 0$ such that for any $v \in V$,
$$
\max\{\|v^+(A)\|,\|(u(r)v)^{+0}(A)\|\} \geq \kappa \|v\|.
$$
\end{lemma}
\par For $\x \in \R^{n-1}$, consider the subgroup $\SL(2,\x) \cong \SL(2,\R) \subset H$. Recall that $\begin{bmatrix}e^t & \\ & e^{-t}\end{bmatrix}$ corresponds to $a(t) \in \SL(2,\x)$ and $\begin{bmatrix}1 & 1 \\ 0 & 1\end{bmatrix}$ corresponds to $u(\x) \in \SL(2, \x)$. Therefore, the above lemma easily implies the following statement:
\begin{corollary}[See Corollary 2.4 of ~\cite{Shah_1}]
\label{corollary_of_lemma_nondivergent}
 Let $V$ be a linear representation of $H=\mathrm{SO}(n,1)$ ($n\geq 2$), fix a norm $\|\cdot\|$ on $V$. Let $A$ be the 
 one-parameter diagonal subgroup of $H$ as above.
Then given a compact set $\mathcal{F}\subset N\backslash \{e\}$, there exists
a constant $\kappa >0$ such that for any $\mathbf{u}\in \mathcal{F}$ and any
$v\in V$,
$$
\max \{\|v^+(A)\|,\|(\mathbf{u}v)^{+0}(A)\|\}\geq \kappa \|v\|,
$$
In particular, for any $t >0$, any $\mathbf{u}\in \mathcal{F}$ and any
$v\in V$,
$$
\max \{\|a(t)v\|,\|a(t)\mathbf{u}v\|\} \geq \kappa \|v\|.
$$
\end{corollary}

\begin{proof}[Proof of Proposition \ref{prop_no_escape_mass}]
  Fix $s_1 \in I$, and a compact subset $\mathcal{F} \subset N$ containing $\{u(\varphi(s)-\varphi(s_1)): s \in I\}$. Let $\kappa >0$ be the constant provided 
  in Corollary \ref{corollary_of_lemma_nondivergent} applied to $\mathcal{F}$. 
\par For any $\epsilon>0$ and $R >0$, by Theorem \ref{non_divergence_theorem}, there exists a 
 compact subset $\mathcal{C}_{\epsilon}\subset G/\Gamma$, such that for any $t >0$, one of the
following holds:
\begin{enumerate}[label=\textbf{A.\arabic*}]
\item \label{case_one} There exist $\gamma \in \Gamma$ and $j\in \{1,\dots , r\}$
such that $$\sup_{s\in I} \| a(t)u(\varphi(s)) g \gamma v_j \| < R,$$
\item \label{case_two} $$|\{ s\in I:  a(t)u(\varphi(s))x \in \mathcal{C}_{\epsilon}\}| \geq (1-\epsilon)|I|.$$
\end{enumerate}
 Fix $s_2 \in I\setminus\{s_1\}$ and denote $\x = s_2 - s_1$. Then because $\Gamma v_i$ is discrete in $V\setminus\{\mathbf{0}\}$, there exists
 a uniform constant $r>0$ such that 
 $$\|u(\varphi(s_1))g \gamma v_i\| \geq r,$$
 for any $v_i$ and $\gamma \in \Gamma$. By Corollary \ref{corollary_of_lemma_nondivergent} applied to $u(\varphi(s_1))g \gamma v_i$
 and $u(\x)$, we get for any $v_i$, $\gamma \in \Gamma$ and $t > 0$,
 $$\sup_{s \in I} \|a(t)u(\varphi(s))g\gamma v_i\| \geq \kappa r.$$
 If we choose $R < \kappa r$, then case \ref{case_one} above can not hold, this shows that 
 $$|\{s\in I: a(t)u(\varphi(s))x \in \mathcal{C}_{\epsilon}\}| \geq (1-\epsilon)|I|.$$
 Let $\mathcal{K}_{\epsilon} = M\mathcal{C}_{\epsilon}$, since $z(s)\in M$, we have 
 $$|\{s\in I: z(s)a(t)u(\varphi(s))x \in \mathcal{K}_{\epsilon}\}| \geq (1-\epsilon)|I|,$$
 i.e., $\mu_t(\mathcal{K}_{\epsilon}) \geq 1-\epsilon$ for all $t >0$.
 \par This completes the proof.
\end{proof}
\subsection{Unipotent invariance of limit measures}
\par We now show that any limit measure $\mu_{\infty}$ of $\{\mu_t: t >0\}$ is invariant under the unipotent subgroup
$W = \{u(t \mathbf{e}_1): t \in \R\}$, which is the main reason to modify the measure from $\lambda_t$ to $\mu_t$:
\begin{proposition}
 \label{prop_invariant_under_unipotent}
 Let $t_i \rightarrow +\infty$ be a sequence such that $\mu_{t_i} \rightarrow \mu_{\infty}$ in weak-$\ast$ topology, then 
 $\mu_{\infty}$ is invariant under $W$-action.
\end{proposition}
\begin{proof}
Given any $f\in C_c(G/\Gamma)$, and $r \in \R$, we have 
$$
\int f(u(r\mathbf{e}_1)x)\dd\mu_{\infty}  =  \lim_{t_i \rightarrow +\infty} \frac{1}{|I|}\int_a^b f(u(r\mathbf{e}_1)z(s)a(t_i)u(\varphi(s))x) \dd s.
$$
We want to argue that 
$$u(r\mathbf{e}_1)z(s)a(t_i)u(\varphi(s)) \approx z(s+r e^{-t_i})a(t_i)u(\varphi(s+r e^{-t_i})).$$
Since $z(s+r e^{-t_i}) \approx z(s)$ for $t_i$ large enough, it suffices to show that 
$$u(r\mathbf{e}_1)z(s)a(t_i)u(\varphi(s)) \approx z(s)a(t_i)u(\varphi(s+r e^{-t_i})).$$
In fact,
$$\begin{array}{cl}
   & z(s)a(t_i)u(\varphi(s+r e^{-t_i})) \\
   = & z(s) a(t_i)u(\varphi(s) + r e^{-t_i} \varphi'(s) + \frac{r^2}{2} e^{-2 t_i} \varphi^{(2)}(s')) \\
   = & z(s)u(r \varphi'(s)) u(\frac{r^2}{2} e^{- t_i} \varphi^{(2)}(s')) a(t_i) u(\varphi(s)).
  \end{array}
$$
By the definition of $z(s)$, we have the above is equal to
$$ u(\frac{r^2}{2} e^{-t_i} z(s)\varphi^{(2)}(s'))u(r \mathbf{e}_1)z(s)a(t_i)u(\varphi(s)).$$
For $t_i$ large enough, $u(\frac{r^2}{2} e^{-t_i} z(s)\varphi^{(2)}(s'))$ is very close to $e$. Therefore, for any $\delta >0$, there exists 
$T>0$, such that for $t_i > T$, 
$$u(r\mathbf{e}_1)z(s)a(t_i)u(\varphi(s)) \overset{\delta}{\approx} z(s+ r e^{-t_i})a(t_i)u(\varphi(s+r e^{-t_i})).$$
Now for any $\epsilon >0$, we choose $\delta>0$ such that whenever $x_1 \overset{\delta}{\approx} x_2$, we have 
$f(x_1) \overset{\epsilon}{\approx} f(x_2)$. Then from the above argument, we have for $t_i > T$,
$$f(u(r \mathbf{e}_1) z(s) a(t_i)u(\varphi(s))x) \overset{\epsilon}{\approx} f(z(s+ r e^{-t_i}) a(t_i) u(\varphi(s+ r e^{-t_i}))x).$$
Therefore,
$$\begin{array}{cl} & \frac{1}{|I|}\int_a^b f(u(r \mathbf{e}_1) z(s) a(t_i)u(\varphi(s))x)\dd s \\
                  \overset{\epsilon}{\approx} & \frac{1}{|I|} \int_a^b f(z(s+ r e^{-t_i}) a(t_i) u(\varphi(s+ r e^{-t_i}))x) \dd s \\
                  = & \frac{1}{|I|} \int_{a+ r e^{-t_i}}^{b+ r e^{-t_i}} f(z(s)a(t_i)u(\varphi(s))x) \dd s.
\end{array}$$
It is easy to see that when $t_i$ is large enough,
$$\frac{1}{|I|} \int_{a+ r e^{-t_i}}^{b+ r e^{-t_i}} f(z(s)a(t_i)u(\varphi(s))x) ds \overset{\epsilon}{\approx} \frac{1}{|I|} \int_{a}^{b} f(z(s)a(t_i)u(\varphi(s))x) \dd s.$$
Therefore, for $t_i$ large enough,
$$\int f(u(r \mathbf{e}_1)x) \dd \mu_{t_i} \overset{2\epsilon}{\approx} \int f(x) \dd \mu_{t_i}.$$
Letting $t_i \rightarrow +\infty$, we have 
$$\int f(u(r \mathbf{e}_1)x) \dd \mu_{\infty} \overset{2\epsilon}{\approx} \int f(x) \dd \mu_{\infty}.$$
Since the above approximation is true for arbitrary $\epsilon >0$, we have that 
$\mu_{\infty}$ is $W$-invariant.
\end{proof}


\section{Ratner's theorem and the Linearization technique}
\label{linearization}

\subsection{Ratner's theorem}
\par Let $\mu_{\infty}$ be a limit measure of $\{\mu_t: t >0\}$, i.e., there exists a sequence $t_i \rightarrow +\infty$ such that 
$\mu_{t_i} \rightarrow \mu_{\infty}$ as $i \rightarrow \infty$ in weak-$\ast$ topology. By Proposition \ref{prop_no_escape_mass} and Proposition \ref{prop_invariant_under_unipotent},
$\mu_{\infty}$ is a probability measure and is $W$-invariant.
\par In this section we will apply Ratner's theorem and the linearization technique to study the property of $\mu_{\infty}$.
\begin{definition}
 Let $\mathcal{L}$ be the collection of analytic subgroups $L < G$ such that $L\cap \Gamma$ is a lattice of $L$. one can prove that
$\mathcal{L}$ is a countable set (see ~\cite{Ratner}).
\par For $L\in \mathcal{L}$, define:
$$N(L,W):= \{g\in G: g^{-1}Wg\subset L\},$$
and 
$$S(L,W):= \bigcup_{L'\in \mathcal{L}, L' \subsetneq L} N(L', W).$$
\end{definition}

\begin{theorem}[see Theorem 1 of ~\cite{Ratner}]
 \label{ratner_theorem} 
Let $\pi: G \rightarrow G/\Gamma$ denote the natural projction. 
Given a $W$-invariant probability measure $\mu$
on $G/\Gamma$
there exists $L\in \mathcal{L}$ such that:
$$\mu(\pi(N(L,W)))>0,$$  
and  
$$\mu(\pi(S(L,W)))=0.$$
Moreover, almost every $W$-ergodic component of $\mu$ on
$\pi(N(L,W))$ is a measure of the form $g\mu_L$ where $g\in
N(L,W)\backslash S(L,W)$, $\mu_L$ is a finite $L$-invariant measure
on $\pi(L)$, and $g\mu_L(E)=\mu_L(g^{-1}E)$ for all Borel sets
$E\subset G/\Gamma$. In particular, if $L\lhd G$, then $\mu$ is
$L$-invariant.
\end{theorem}
\par If $\mu_{\infty}=\mu_G$, then there is nothing to prove. So we may assume $\mu_{\infty}\neq \mu_G$. Then by Ratner's Theorem, 
there exists $L \in \mathcal{L}$ such that $\mu_{\infty}(\pi(N(L,W))) > 0$ and $\mu_{\infty}(\pi(S(L,W))) =0$. 
\subsection{The linearization technique}
Now we start to apply the linearization technique.
\par We start with some basic notations:
\begin{definition}
\par Let $V$ be the finitely dimensional representation of $G$ defined as in Definition \ref{def_representation_G}, for $L \in \mathcal{L}$,
we choose a basis $\mathfrak{e}_1, \mathfrak{e}_2, \dots, \mathfrak{e}_l$ of the Lie algebra $\mathfrak{l}$ of $L$, and define 
$$p_L = \wedge_{i=1}^l \mathfrak{e}_i \in V.$$
Define 
$$\Gamma_L := \left\{\gamma\in \Gamma: \gamma p_L = \pm p_L \right\}.$$
From the action of $G$ on $p_L$, we get a map:
$$\begin{array}{l} \eta:  G \rightarrow V,  \\
                     g \mapsto g p_L .
   \end{array}
$$
We define $\mathcal{A}$ to be the Zariski closure of $\eta(N(L,W))$. and for any compact subset $\mathcal{D} \subset \mathcal{A}$, we define
$$S(\mathcal{D}) := \left\{g \in N(L,W):  \eta(g\gamma) \in \mathcal{D} \text{ for some } \gamma \in \Gamma\setminus \Gamma_L\right\}.$$
\end{definition}

Concerning $S(\mathcal{D})$, we have the following important proposition:
\begin{proposition}[see Proposition 4.5 of ~\cite{Shah_1}]
\label{injective_prop}
 $S(\mathcal{D}) \subset S(L,W)$ and
$\pi(S(\mathcal{D}))$ is closed in $G/\Gamma$. Moreover, for any
compact set $\mathcal{K} \in G/\Gamma \backslash
\pi(S(\mathcal{D}))$, there exists some neighborhood $\Phi$ of
$\mathcal{D}$ in $V$ such that, for any $g\in G$ and $\gamma_1,
\gamma_2 \in \Gamma$, if $\pi(g)\in \mathcal{K}$ and $\eta(g
\gamma_i) \in \Phi$, $i=1,2$, then $\eta(\gamma_1)=\pm
\eta(\gamma_2)$.
\end{proposition}
\begin{proposition}[see Proposition 4.6 of ~\cite{Shah_1}]
\label{relative_small prop} Given a symmetric compact set
$\mathcal{C}\subset\mathcal{A}$ and $\epsilon > 0$, there exists a
symmetric compact set $\mathcal{D}\subset \mathcal{A}$ containing
$\mathcal{C}$ such that, given a symmetric neighborhood $\Phi$ of
$\mathcal{D}$ in $V$, there exists a symmetric neighborhood $\Psi$
of $\mathcal{C}$ in $V$ contained in $\Phi$ such that for any $t >0$, 
for any $v \in V$, and for any interval
$J\subset I$, one of the following holds:
\begin{enumerate}[label=\textbf{S.\arabic*}]
\item \label{1} $a(t)u(\varphi(s))v \in \Phi$ for all $s\in J$,
\item \label{2} $|\{s\in J: a(t)u(\varphi(s))v \in \Psi \}|\leq \epsilon |\{s\in J: a(t)u(\varphi(s))v \in \Phi\}|$.
\end{enumerate}
\end{proposition}
\begin{remark}
 The proof is similar to Theorem \ref{non_divergence_theorem}, and also follows from the fact that all coordinate functions of
 $a(t)u(\varphi(\cdot))v$ are 
 $(C, \alpha)$-good for some constants $C >0$ and $\alpha >0$.
\end{remark}
\subsection{Linear algebraic condition on the curve}
\par The following proposition is the 
aim of this section.
\begin{proposition}
 \label{prop_algebraic_condition}
 There exists a $\gamma \in \Gamma$ such that
 $$u(\varphi(s))g\gamma p_L \in V^{-0}(A), \text{ for all } s \in I .$$
 
\end{proposition}

\begin{proof}
 Take a compact subset $C \subset N(L,W))\setminus S(L,W)$ such that $\mu_{\infty}(\pi(C)) >c_0 >0$ for some constant $c_0$. Define 
 $\mathcal{C} := \eta(C)\cup - \eta(C)$, then $\mathcal{C} \subset \mathcal{A}$ is a compact subset. Choose a 
 compact subset $\mathcal{K}\subset G/\Gamma\setminus \pi(S(L,W))$ containing $\pi(C)$ in its interior. Applying Proposition \ref{relative_small prop},
 we have that for any $\epsilon >0$, there exists a symmetric compact subset $\mathcal{D} \subset \mathcal{A}$ containing $\mathcal{C}$ such that 
 the conclusion of Proposition \ref{relative_small prop} for $\mathcal{C}$, $\mathcal{D}$ and $\epsilon$. Let us fix $\epsilon \in (0, \frac{c_0}{2})$.
 Applying Proposition \ref{injective_prop} to 
 $\mathcal{D}$ and $\mathcal{K}$, we have that there exists an open neighborhood $\Phi$ of $\mathcal{D}$ such that the conclusion of Proposition 
 \ref{injective_prop} holds. 
 Choose a neighborhood
 $\Psi$ of $\mathcal{C}$ according to Proposition \ref{relative_small prop}. Since $M \cong \SO(n-1)$ is compact, we could enlarge $\mathcal{C}$, $\mathcal{D}$, $\mathcal{K}$, $\Phi$ and $\Psi$ to make them invariant under the action of $M$.
 \par Recall that $\mu_{\infty} = \lim_{i \rightarrow \infty} \mu_{t_i}$. We claim that for each $t_i$, there exists $\gamma_{t_i} \in \Gamma$ such that 
 $$a(t_i)u(\varphi(I))g \gamma_{t_i} p_L \subset \Phi.$$
 For contradiction, suppose it is not the case, i.e., 
 for all $\gamma \in \Gamma$, case \ref{1} in Proposition \ref{relative_small prop} does not hold for 
 $v = g\gamma p_L$ and $J=I$. Define
 $$J_{t_i} := \left\{ s\in I: a(t_i)u(\varphi(s))x \in \mathcal{K} : a(t_i)u(\varphi(s))g\Gamma p_L \cap \Psi \neq \emptyset  \right\}.$$
 Since $\mathcal{K}$ and $\Psi$ are $M$-invariant, we have 
 $$J_{t_i} = \left\{ s\in I: z(s)a(t_i)u(\varphi(s))x \in \mathcal{K} : z(s)a(t_i)u(\varphi(s))g\Gamma p_L \cap \Psi \neq \emptyset  \right\}.$$
 Note that $\mu_{t_i}$ denotes the normalized Lebesgue measure on the curve $\{z(s)a(t_i) u(\varphi(s))x : s \in I\}$ and $\mu_{t_i} \rightarrow \mu_{\infty}$ as $i \rightarrow \infty$. Since $\mu_{\infty}(\pi(C)) > c_0$, for $i$ large enough, $|J_{t_i}| > c_0|I|$ . 
 \par For any $s \in J_{t_i}$, by Proposition \ref{injective_prop}, up to $\pm$ sign, there exists unique $\gamma(s) p_L$ such that 
 $a(t_i)u(\varphi(s))g\gamma(s) p_L \in \Psi$, let $I_{\gamma(s)}$ be the maximal interval $I$ containing $s$ such that 
 $$a(t_i)u(\varphi(I))g \gamma(s) p_L \subset \Phi.$$
 From Proposition \ref{injective_prop} we know that there is no other $\gamma' p_L$ other than $\pm \gamma(s) p_L$ 
 and $s \in I_{\gamma(s)}\cap J_{t_i}$ such that
 $$a(t_i)u(\varphi(s))g \gamma' p_L \in \Psi.$$
 Therefore $J_{t_i}$ is covered by at most countably many intervals $I_{\gamma(s)}$'s which covers the whole interval $I$ at most twice, namely,
 every point belongs to at most two different intervals (this is because for any $s_1< s_2 \in J_{t_i}$, then from the above argument, the intersection 
 $I_{\gamma(s_1)}\cap I_{\gamma(s_2)} \subset (s_1, s_2)$).
 Moreover, because case \ref{1} in Proposition \ref{relative_small prop} does not hold, we have that \ref{2} must hold, i.e., 
 $$|J_{t_i} \cap I_{\gamma(s)}| < \epsilon |I_{\gamma(s)}|.$$
 This shows that 
 $$|J_{t_i}| < 2\epsilon |I|$$
 which contradicts to the fact that $|J_{t_i}| > c_0|I|$. This proves the claim.
 \par Since $\Gamma p_L$ is discrete in $V$, one of the following will happen:
 \begin{enumerate}
  \item $\|\gamma_{t_i} p_L\| \rightarrow +\infty$ as $i \rightarrow \infty$.
  \item $\gamma_{t_i} p_L$ remains the same for all large $i$.
 \end{enumerate}
If case (1) happens, define a unit vector $v_{t_i} = \frac{\gamma_{t_i} p_L}{\|\gamma_{t_i} p_L\|}$ for each $i$, 
then from 
$$a(t_i)u(\varphi(I)) \gamma_{t_i} p_L \subset \Phi$$
we have there is a constant $R$ such that 
$$\sup_{s \in I}\|a(t_i)u(\varphi(s))v_{t_i}\| \leq \frac{R}{\|\gamma_{t_i} p_L\|} \rightarrow 0.$$
Suppose $v_{t_i} \rightarrow v_{\infty}$ passing to some subsequence, then we have 
$$\sup_{s \in I} \|a(t_i)u(\varphi(s))v_{\infty}\| \rightarrow 0,$$
as $i \rightarrow +\infty$. This is impossible according to Corollary \ref{corollary_of_lemma_nondivergent}. 
Therefore $\gamma_{t_i} p_L = \gamma p_L$ remains the same for all large $i$. This means that for all $i >0$,
$$\sup_{s \in I}\|a(t_i)u(\varphi(s))g \gamma p_L\| \leq R.$$
This implies that for $v = g \gamma p_L$,
$$u(\varphi(s))v \in V^{-0}(A).$$
\par This completes the proof.
\end{proof}
\par Our goal is to get an explicit geometric condition of $\varphi(I)$ from the above linear algebraic condition.


\section{Basic lemma on $\mathrm{SL}(2,\R)$ representations}
\label{basic lemma proof}
\par This section is devoted to the proof of the following basic lemma concerning representations of $\mathrm{SL}(2,\R)$, which is crucial to 
analyze the condition $u(\varphi(s))v \in V^{-0}(A)$ we get in the previous section:
\begin{lemma}
 \label{basic_lemma}
 Let $V$ be a finite dimensional linear representation of $\mathrm{SL}(2,\R)$. Denote 
 $$
 A := \left\{a(t):= \begin{bmatrix}e^t & \\ & e^{-t}\end{bmatrix}: t \in \R\right\},
 $$
 and 
 $$
 U := \left\{u(s) := \begin{bmatrix}1 & s \\ 0 & 1\end{bmatrix} : s \in \R\right\}.
 $$
 Suppose there is a nonzero vector $v \in V^{-0}(A)$ satisfying 
 $$u(r)v \in V^{-0}(A),$$
 for some $r \in \R$, then $(u(r)v)^0(A) = \sigma v^0(A),$
 where $\sigma$ denotes the matrix
 $$\sigma = \begin{bmatrix}0 & -1 \\ 1 & 0\end{bmatrix} \in \SL(2,\R).$$
\end{lemma}
\begin{proof}
In the Lie algebra $\mathfrak{sl}(2,\R)$, let us denote 
$$\mathfrak{n} = \begin{bmatrix}0 & 1 \\ 0 & 0\end{bmatrix},$$
and 
$$\mathfrak{h} = \begin{bmatrix}1 & 0 \\ 0 & -1\end{bmatrix}.$$
\par We decompose $V$ into direct sum of irreducible subspaces
of $\mathrm{SL}(2,\mathbb{R})$, say $V=\bigoplus_{i=1}^m V_i$. For each irreducible component $V_i$, we have the corresponding
projection $q_i: V\rightarrow V_i$. Since every irreducible
representation of $\mathrm{SL}(2,\mathbb{R})$ can be written as
$\mathrm{Span}_{\mathbb{R}}\{w_0, w_1,\dots , w_l\}$, and for $w_k$, $\mathfrak{h} w_k=
(l-2k)w_k$ and $\mathfrak{n} w_k=k w_{k-1}$, then
$u(r)w_k=\sum_{j=0}^k \binom{k}{j} r^{k-j} w_j$. If $l$ is odd,
then $V_i$ does not have contribution to $V^0(A)$, so we just consider
the case when $l=2p$ is even. Let $v_i$ be the image of $v$ under
the projection $q_i: V\rightarrow V_i$, then since $v\in V^- + V^0$
and $u(r)v \in V^{-0}(A)$, for $v_i$ we also have $v_i \in V^- +
V^0$ and $u(r)v_i \in V^{-0}(A)$. Then we can assume that
$v_i=\sum_{k=p}^{2p} c_k w_k$ ($w_p\in V^0(A)$, $w_k \in V^{-}(A) \text{ for
} k> p$). 
\par Let us calculate the $w_p$ coefficient of $u(r)v_i$.
According to the description of $V_i$ above, for $j\leq p$, the $w_j$ coefficient of
$u(r)v_i$ is $\sum_{k=p}^{2p} \binom{k}{j} r^{k-j} c_k$. Then
since $u(r)v_i \in V^{-0}(A)$, we have for $j> p$, its $w_j$
coefficient, which is $\sum_{k=p}^{2p} \binom{k}{j} r^{k-j} c_k$,
equals $0$. Thus we get a series of equations:
$$
\begin{array}{c}
\sum_{k=p}^{2p} \binom{k}{p-1} r^{k-p+1} c_k= 0 \\
   \\              \\
\sum_{k=p}^{2p}\binom{k}{p-2} r^{k-p+2} c_k= 0 \\   \\      \\
 \vdots
\\   \\    \\ \sum_{k=p}^{2p}\binom{k}{0} r^{k}  c_k =0.
\end{array}
$$
From these equations we want to find out its $w_p$ coefficient, say
$\sum_{k=p}^{2p} \binom{k}{p} r^{k-p} c_k$.
\par Define 
  $$F(x) = \sum_{k=p}^{2p} r^{k-p} c_k x^k,$$
  then from the above equations we could easily deduce that 
  $$F^{(i)}(1)= 0,$$
  for $i=0,1,2,\dots, p-1$, here $F^{(i)}(x)$ denotes the $i$th derivative of $F(x)$. Therefore, the power of factor $(x-1)$ in $F(x)$ is at least 
  $p$. Combined with the fact that the power of factor $x$ in $F(x)$ is at least $p$, this implies
  $$F(x)= C x^p (x-1)^p,$$
  for some constant $C$.  Therefore the coefficient of $w_p$ in $u(r)v_i$ is 
  $$\frac{1}{p!} F^{(p)}(1)= \frac{1}{p!} p! C = C.$$
  On the other hand, from $F(x)= C x^p(x-1)^p$, the coefficient of $x^p$ of $F(x)$, $c_p = (-1)^p C$. 
  Note that the Weyl element $\sigma = \begin{bmatrix}0 & -1 \\ 1 & 0\end{bmatrix}$ sends $c_p w_p$ to $(-1)^p c_p w_p$. 
  Therefore we get 
  $$(v_i)^0(A) = \sigma (u(r)v_i)^0(A).$$
  By taking the summation over all irreducible components $V_i$ of $V$, we get 
  $$v^0(A) = \sigma (u(r)v)^0(A).$$
  This completes the proof.
\end{proof}


\section{Conclusion}
\label{conclusion}
To conclude Theorem \ref{goal_thm}, it suffices to prove Theorem \ref{main_theorem}.
\begin{proof}[Proof of Theorem \ref{main_theorem}]
We start with the condition that there exists a nonzero vector $v \in V$ such that 
$$u(\varphi(s))v \in V^{-0}(A) \text{ for all } s \in I.$$

\par Fix $s_0 \in I$. By replacing $v$ by $ u(\varphi(s_0))v$, $u(\varphi(s))$ by $\psi(s) :=\varphi(s) - \varphi(s_0) $, we may assume that $v \in V^{-0(A)}$, $\psi(s_0) =\mathbf{0}$, and $u(\psi(s))v \in V^{-0}(A)$ for all $s \in I$.

\par  We claim that $(v)^0(A)$
is fixed by $\mathrm{SL}(2,\varphi^{(1)}(s_0))$. In fact, since $u(\varphi(s))v \in
V^{-0}(A)$ for any $s$, by taking derivative at $s_0$, we have that
$\varphi^{(1)}(s_0) u(\varphi(s_0))v = \varphi^{(1)}(s_0) v \in V^{-0}(A)$ (here $\varphi^{(1)}(s_0) \in \R^{n-1} \cong \mathfrak{n}$ acts on $V$ by the derived Lie algebra action). Therefore, $\varphi^{(1)}(s_0) (v)^0(A) = \mathbf{0}$
since otherwise it will belong to $V^{+}(A)$ (the eigenvalues of eigenvectors for $A$ are increased by $\varphi^{(1)}(s_0)$). This implies that
$(v)^0(A)$ is fixed by $\{u(r \varphi^{(1)}(s_0)): r \in \R\}$. Since by definition, $(v)^0(A)$ is
also fixed by $A$, we have that it is fixed by the whole
$\mathrm{SL}(2,\varphi^{(1)}(s_0))$. This proves the claim. In fact, the same argument shows that for any $s \in I$,
$(u(\varphi(s))v)^{0}(A)$ is fixed by $\SL(2, \varphi^{(1)}(s))$.
\par  With respect to
the stabilizer of
$(v)^0(A)$ in $N$, we make the following claim:
\begin{claim} 
The set $S =\{\mathbf{x} \in \mathbb{R}^{n-1}:
u(\mathbf{x})(v)^0(A) =(v)^0(A) \}$ is a proper subspace, in other
words, there exists some vector $\mathbf{w}\in \mathbb{R}^{n-1}$
such that if $\mathbf{x} \in S$, then $\langle \mathbf{x}, \mathbf{w}\rangle=0$,
where $\langle \cdot , \cdot\rangle$ denotes the standard inner product on $\R^{n-1}$. 
\end{claim} 
\par
\begin{proof}[Proof of the claim]
\par In fact, it is easy to see that $S$ is a subspace of
$\mathbb{R}^{n-1}$. We only need to show it is a proper subspace. 
\par Suppose not, i.e., $S=\mathbb{R}^{n-1}$, then
$u(\mathbf{x})(v)^0(A)=(v)^0(A)$ for all $\mathbf{x} \in
\mathbb{R}^{n-1}$. Since $(v)^0(A)$ is also fixed by $A$, we $(v)^0(A)$ is fixed by the whole $\mathrm{SO}(n,1)$.  
\par Now we consider $(v)^{-}(A) \in V^-(A)$. 
If $(v)^{-}(A) \neq \mathbf{0}$, then for all $s \in I$,
$$\begin{array}{rcl} u(\varphi(s))(v)^{-}(A) & = & u(\varphi(s) )v - u(\varphi(s) )(v)^0(A) \\
 & = & u(\varphi(s))v - (v)^0(A)  \in V^{-0}(A)\end{array}$$
 This is impossible by Lemma \ref{lemma_non_divergent}. Therefore, 
$(v)^{-}(A)=0$, i.e., $v= (v)^{0}(A) \in V^0(A)$. Then $p_L$
is fixed by the action of
$\gamma^{-1}g^{-1}\mathrm{SO}(n,1)g\gamma$ . Thus
$$
\begin{array}{rcl}
\Gamma p_L & = & \overline{\Gamma p_L} \text{ since } \Gamma p_L
 \text{ is discrete} \\
 & = & \overline{\Gamma \gamma^{-1} g^{-1}\mathrm{SO}(n,1)g\gamma
 p_L} \\ & = & \overline{\Gamma g^{-1}\mathrm{SO}(n,1)g\gamma
 p_L} \\ & = & G g\gamma p_L \text{ since } \overline{\mathrm{SO}(n,1)g
 \Gamma}=G \\
  & = & G p_L.
 \end{array}
$$
This implies $G_0 p_L = p_L$ where $G_0$ is the connected component
of $e$. In particular, $\gamma^{-1} g^{-1}\mathrm{SO}(n,1) g
\gamma\subset G_0$ and $G_0 \subset N^1_{G}(L)$. By ~\cite[Theorem
2.3]{Shah_2}, there exists a closed subgroup $H_1 \subset
N^1_{G}(L)$ containing all $\mathrm{Ad}$-unipotent one-parameter
subgroups of $G$ contained in $N^1_{G}(L)$ such that $H_1 \cap
\Gamma$ is a lattice in $H_1$ and $\pi (H_1)$ is closed. If we put
$F= g\gamma H_1 \gamma^{-1} g^{-1}$, then $\mathrm{SO}(n,1)\subset
F$ since $\mathrm{SO}(n,1)$ is generated by it unipotent
one-parameter subgroups. Moreover, $Fx =g \gamma \pi(H_1)$ is closed
and admits a finite $F$-invariant measure. Then since
$\overline{\mathrm{SO}(n,1)x}=G/\Gamma$, we have $F=G$. This implies
$H_1 = G$ and thus $L \lhd G$. Therefore $N(L,W)=G$. In particular,
$W\subset L$, and thus $L\cap \mathrm{SO}(n,1)$ is a normal subgroup
of $\mathrm{SO}(n,1)$ containing $W$. Since $\mathrm{SO}(n,1)$ is a
simple group, we have $\mathrm{SO}(n,1) \subset L$. Since $L$ is a
normal subgroup of $G$ and $\pi(L)$ is a closed orbit with finite
$L$-invariant measure, every orbit of $L$ on $G/\Gamma$ is also
closed and admits a finite $L$-invariant measure, in particular,
$Lx$ is closed. But since $\mathrm{SO}(n,1)x$ is dense in
$G/\Gamma$, $Lx$ is also dense. This shows that $L=G$, which
contradicts to our hypothesis that $\mu_{\infty} \neq \mu_G$. This proves the claim. 
\end{proof}
\par Put $\varphi(s)=r(s)k(s)\mathbf{e}_1$, where $r(s) \in \mathbb{R}$
and $k(s)\in \mathrm{SO}(n-1)$. Recall that we have $u(k \mathbf{x})= k u(\mathbf{x}) k^{-1}$
for any $k \in M \cong \mathrm{SO}(n-1)$ and $\mathbf{x} \in
\mathbb{R}^{n-1}$.
\par We rewrite $u(\varphi(s))v$ as follows:
$$
\begin{array}{rl}
u(\varphi(s))v & = u(r(s)k(s) \mathbf{e}_1)v \\
& = k(s)u(r(s)\mathbf{e}_1)k^{-1}(s)v.
\end{array}
$$
Since $k(s) \in Z(A)$ preserves the eigenspaces of $A$, we have $k^{-1}(s)v
\in V^{-0}(A)$ and $u(r(s)\mathbf{e}_1)k^{-1}(s)v$.
Define
$$E =\begin{bmatrix} -1 & & & \\ & 1 & & \\
 & & \ddots &  \\ & & & 1 \end{bmatrix}.$$  
 Applying Lemma \ref{basic_lemma} 
with $v$ replaced by $k^{-1}(s)v$, and $\mathrm{SL}(2,\R)$ replaced by $\mathrm{SL}(2,\mathbf{e}_1)$, we have $(u(r(s) \mathbf{e}_1)k^{-1}(s)v)^0(A) = \mathfrak{J}
(k^{-1}(s)v)^0(A)= \mathfrak{J} k^{-1}(s)v^0(A)$, where 
$$\mathfrak{J} = \begin{bmatrix} & & 1 \\ & E & \\ 1 &  & \end{bmatrix}$$
corresponds to $\begin{bmatrix}0 & -1 \\ 1 & 0 \end{bmatrix}$ in $\mathrm{SL}(2,\mathbf{e}_1)$.
 Thus
 $$
 \begin{array}{rl}
 (u(\varphi(s))v)^0(A) & =(k(s)u(r(s) \mathbf{e}_1)k^{-1}(s)v)^0(A) \\
 & =k(s)(u(r \mathbf{e}_1)k^{-1}(s)v)^0(A) \\ 
 & = k(s)\mathfrak{J}k^{-1}(s)v^0(A).
 \end{array}
 $$
 \par From
 our previous claim,
 $(u(\varphi(s))v)^0(A)=k(s)\mathfrak{J}k^{-1}(s)v^0(A)$ is fixed by $\mathrm{SL}(2, \varphi^{(1)}(s))$. Let $h(s) := k(s)\mathfrak{J}k^{-1}(s)$, then we have that $v^0(A)$ is fixed by $h^{-1}(s)
 \mathrm{SL}(2,\varphi^{(1)}(s))h(s)$. In particular, it is fixed by $h^{-1}(s)u^-(\varphi^{(1)}(s))
 h(s)$. By direct calculation, 
 $$
 \begin{array}{rl}h^{-1}(s)u^-(\varphi^{(1)}(s))h(s) &= k(s)\mathfrak{J}k^{-1}(s)u^-(\varphi^{(1)}(s))k(s)\mathfrak{J}
 k^{-1}(s) \\ & = u(k(s)Ek^{-1}(s)\varphi^{(1)}(s)).\end{array}
 $$
 \par
 By our previous claim, we have that $\langle k(s)Ek^{-1}(s)\varphi^{(1)}(s),
 \mathbf{w} \rangle=0$.
 \par From $\varphi(s)=r(s)k(s)\mathbf{e}_1$, we have
$\varphi^{(1)}(s)
 =r^{(1)}(s)k(s)\mathbf{e}_1+r(s)k^{(1)}(s)\mathbf{e}_1$,
 so
 $$
 \begin{array}{rl}
 k(s)Ek^{-1}(s)\varphi^{(1)}(s) & = k(s)Ek^{-1}(s) (r^{(1)}(s)k(s)\mathbf{e}_1+r(s) k^{(1)}(s)\mathbf{e}_1)
 \\  & = k(s)Ek^{-1}(s) r^{(1)}(s)k(s)\mathbf{e}_1+ k(s)Ek^{-1}(s)r(s)k^{(1)}(s)\mathbf{e}_1 \\
  & = r^{(1)}(s)k(s)E \mathbf{e}_1 + r(s)k(s)E k^{-1}(s) k^{(1)}(s)\mathbf{e}_1 \\
  & = -r^{(1)}(s)k(s)\mathbf{e}_1 + r(s)k(s)Ek^{-1}(s)k^{(1)}(s)\mathbf{e}_1 \end{array}
  $$
  \par 
  For $r(s)k(s)E k^{-1}(s)k^{-1} k^{(1)}(s)\mathbf{e}_1$, we put $k(s)=[a_1,\dots ,
  a_{n-1}]$, where $a_i \in \mathbb{R}^{n-1}$ are column vectors. It is clear that $k^{-1}(s)=k(s)^t$
  since $k(s)\in \mathrm{SO}(n-1)$. Then $k^{(1)}(s)=[a^{(1)}_1,\dots ,
  a^{(1)}_{n-1}]$, so $k^{(1)}(s)\mathbf{e}_1=a^{(1)}_1$. Then, the first
  coordinate of $k^{-1}(s)k^{(1)}(s)\mathbf{e}_1$ is equal to $a_1^t a^{(1)}_1=\langle a_1,
  a^{(1)}_1\rangle$. Given $k(s)\in \mathrm{SO}(n-1)$, we have $\langle a_1, a_1\rangle=
  1$, by taking derivative, we get $\langle a_1, a^{(1)}_1\rangle=0$, i.e., the
  first coordinate of $k^{-1}(s)k^{(1)}(s)\mathbf{e}_1$ is zero. It follows
  that $k^{-1}(s)k^{(1)}(s)\mathbf{e}_1$ is fixed by $E$, thus $k(s)Ek^{-1}(s)k^{(1)}(s)\mathbf{e}_1
  =k(s)k^{-1}(s)k^{(1)}(s)\mathbf{e}_1=k^{(1)}(s)\mathbf{e}_1$. Combining them, we
  have
  $$\begin{array}{rl}k(s)Ek^{-1}(s)\varphi^{(1)}(s) &= -r^{(1)}(s)k(s)\mathbf{e}_1+r(s) k^{(1)}(s)\mathbf{e}_1 \\
    & = \varphi^{(1)}(s)-2 r^{(1)}(s)k(s)\mathbf{e}_1 \\
    & = \varphi^{(1)}(s)- 2r^{(1)}(s)\varphi(s)/r(s) .
    \end{array}
    $$
    Therefore, $\langle \varphi^{(1)}(s)- 2 r^{(1)}(s) \varphi(s)/r(s),\mathbf{w}\rangle=0$
    for any $s\in I$. This means $\frac{\dd}{\dd s }
    (\frac{\langle \varphi(s),\mathbf{w}\rangle}{r^2(s)})=0$, which implies
    $\frac{\langle \varphi(s),\mathbf{w}\rangle}{r^2(s)}=C$ for a constant $C$. For $C=0$, this
    gives an equation of hyperplane, for $C \neq 0$, this gives an
    equation for a subsphere.
    \par This completes the proof.
\end{proof}

\bibliography{reference}{}
\bibliographystyle{plain}
\end{document}